\documentclass[12pt]{amsart} 
\usepackage{amssymb}
\usepackage{fullpage}
\usepackage[colorlinks=true]{hyperref}
\usepackage{tikz-cd}

\newcommand{\bv}{\mathbf{v}}
\newcommand{\bw}{\mathbf{w}}

\newcommand{\bbF}{\mathbb{F}}
\newcommand{\bbK}{\mathbb{K}}

\newcommand{\bbQ}{\mathbb{Q}}
\newcommand{\bbR}{\mathbb{R}}
\newcommand{\bbZ}{\mathbb{Z}}

\newcommand{\cC}{\mathcal{C}}

\font\cyr=wncyr10
\newcommand{\Sha}{\hbox{\cyr X}}

\DeclareMathOperator{\coker}{coker}

\DeclareMathOperator{\Image}{im}

\DeclareMathOperator{\myMod}{mod}
\DeclareMathOperator{\ord}{ord}

\DeclareMathOperator{\rank}{rank}

\numberwithin{equation}{section}

\newtheorem{theorem}{Theorem}[section]

\newtheorem{corollary}[theorem]{Corollary}
\newtheorem{lemma}[theorem]{Lemma}
\newtheorem{proposition}[theorem]{Proposition}

\theoremstyle{definition}
\newtheorem{example}[theorem]{Example}

\newtheorem*{remark-nonum}{Remark}
\newtheorem*{remarks-nonum}{Remarks}

\begin{document}

\title{Exact classification of elliptic curves $y^{2}=x^{3}-pqx$ with rank $0$ and trivial $\Sha[2]$}

\author{Arkabrata Ghosh and Paul M. Voutier}

\date{}

\begin{abstract}
For the elliptic curves $E_{p,q}: y^{2}=x^{3}-pqx$ where $p$ and $q$ are distinct
odd primes, we establish necessary and sufficient conditions under which
$\rank E_{p,q}(\bbQ)$ and $\dim_{\bbF_{2}} \Sha \left( E_{p,q}/\bbQ \right)[2]$ are both
$0$. We do so via a similar characterisation of when the Selmer groups
associated with the degree-$2$ isogeny $\phi$ and its dual $\widehat{\phi}$
are both of minimal size, along with results about a cokernel that arises from a
related exact sequence.
\end{abstract}


\maketitle

\section{Introduction}
\label{sect:intro}

Two of the most important arithmetic objects in the study of elliptic curves are
the Mordell-Weil group and the Tate-Shafarevich group. The Selmer groups associated
with elliptic curves are very useful tools for studying these objects. The work of
Bhargava and Shankar \cite{B-S1, B-S2, B-S3, B-S4}
shows the power of the study of Selmer groups for obtaining striking results
about the ranks of Mordell-Weil groups.

In Proposition~X.6.2(c) of \cite{S09}, Silverman characterises the values of
$\rank E_{p}(\bbQ)+\dim_{2} \Sha \left( E_{p}/\bbQ \right)[2]$ for the elliptic
curves $E_{p}: y^{2}=x^{3}+px$ where $p$ is an odd prime. He obtained this result
via a study of the Selmer groups associated with the isogeny $\phi$ of degree $2$
and its dual $\widehat{\phi}$.

Here, and throughout this paper, $\dim_{2}$ denotes the dimension of an $\bbF_{2}$-vector space,
following Silverman in Section~X.6 of \cite{S09}.

When $p$ is replaced by a product of primes, the corresponding descent analysis
becomes more complicated. First, the possibilities for these Selmer groups
increase, along with the conditions
required for each possibility. In addition, a cokernel that we denote by $\cC_{S}$
in Proposition~\ref{prop:2.1} below is trivial for Silverman's curves, but is not
always trivial for our elliptic curves. As a consequence,
$\rank E(\bbQ)+\dim_{2} \Sha \left( E/\bbQ \right)[2]$ no longer depends only
on the sizes of these Selmer groups. So although several authors since
Silverman have extended his results for certain families of elliptic curves (see,
for instance, \cite{Gh,Yo}, as well as many others), none have obtained results
providing necessary and sufficient conditions for the values.

In this paper, we will consider the elliptic curves $E_{p,q}$ given by
\begin{equation}
\label{eq:1.1}
E_{p,q}: y^{2}= x^{3}-pqx,
\end{equation}
where $p$ and $q$ are distinct odd primes.
We study the $\phi$- and $\widehat{\phi}$-Selmer groups of these curves. These
groups, which we denote by $S^{(\phi)} \left( E_{p,q}/\bbQ \right)$ and
$S^{(\widehat{\phi})} \left( \widehat{E_{p,q}}/\bbQ \right)$,
are defined in Section~\ref{sect:2-iso} below. We determine
precisely when these two Selmer groups are both as small as possible.


We restrict to the case when both these Selmer groups are as small as possible
mainly because in this case, we are able to
determine not just the sum $\rank E_{p,q}(\bbQ)+\dim_{2} \Sha \left( E_{p,q}/\bbQ \right)[2]$
(which is $0$), but also the two summands themselves, both of which are $0$.

Furthermore, by our study of $\cC_{S}$ we are also able to characterise
precisely when both $\rank E_{p,q}(\bbQ)$ and $\dim_{2} \Sha \left( E_{p,q}/\bbQ \right)[2]$
are $0$. To the best of our knowledge, this is the first such result beyond
Silverman's work to provide such a characterisation.

Lastly, our study of $\cC_{S}$ along with an investigation of the other eight
possibilities for the sizes of the two Selmer groups would also permit us to
characterise precisely when
$\rank E_{p,q}(\bbQ)+\dim_{2} \Sha \left( E_{p,q}/\bbQ \right)[2]$
takes on each of its possible values. For those with sufficient stamina and
interest, that would fully generalise Silverman's result for his elliptic curves
$E_{p}$ to our elliptic curves $E_{p,q}$.

\begin{theorem}
\label{thm:rk0}
Suppose $p$ and $q$ are distinct odd primes.

$S^{(\phi)} \left( E_{p,q}/\bbQ \right) = \{1, pq \}$ and
$S^{(\widehat{\phi})} \left( \widehat{E_{p,q}}/\bbQ \right) = \{1, -pq \}$
if and only if one of the following conditions holds.

\noindent
{\rm (i)} $p,q \equiv 1 \pmod{4}$, $pq \equiv 13 \pmod{16}$ and $\left( \dfrac{p}{q} \right)=-1$.

\noindent
{\rm (ii)} $p \not\equiv q \pmod{4}$, $pq \equiv 3 \pmod{8}$ and $\left( \dfrac{p}{q} \right)=-1$.

\noindent
{\rm (iii)} $p \equiv q \equiv 3 \pmod{4}$,
$\left( \dfrac{p}{q} \right)=1$, $q \equiv 3 \pmod{8}$ and $p-q \equiv 8, 12 \pmod{16}$.

\noindent
{\rm (iv)} $p \equiv q \equiv 3 \pmod{4}$, $\left( \dfrac{p}{q} \right)=-1$,
$p \equiv 3 \pmod{8}$ and $q - p \equiv 8, 12 \pmod{16}$.
\end{theorem}

As mentioned above, we are able to use this theorem to characterise precisely
when the non-torsion part of the Mordell-Weil group and the $2$-primary subgroup
of the Tate-Shafarevich group are both trivial for our curves, $E_{p,q}$.

\begin{theorem}
\label{thm:rk0-sha1}
The rank of $E_{p,q}(\bbQ)$ is $0$ and $\dim_{2} \Sha \left( E_{p,q}/\bbQ \right)[2]=0$
if and only if one of the conditions in Theorem~$\ref{thm:rk0}$ holds.
\end{theorem}

Since $\dim_{2} \Sha \left( E_{p,q}/\bbQ \right)[2]=0$ if and only if
$\Sha \left( E_{p,q}/\bbQ \right)\left[ 2^{\infty} \right]$ is trivial, the theorem
also holds with this latter condition, justifying the statement before
Theorem~\ref{thm:rk0-sha1}.

\begin{example}
For $p \equiv 7 \pmod{8}$ and $q \equiv 5 \pmod{8}$ with $\left( \dfrac{p}{q} \right)=-1$,
we have $pq \equiv 3 \pmod{8}$, so condition~(ii) here is satisfied. In this way,
we obtain Theorem~1.1 in \cite{Gh} as a special case of our result. In fact, our
results also strengthen Theorem~1.1 there, as we show that the $2$-primary subgroup
of the Tate-Shafarevich group is always trivial for such curves.
\end{example}

\begin{example}
It is very easy to find examples where the sizes of the $\phi$- or $\widehat{\phi}$-Selmer
groups are larger, but the rank of the Mordell-Weil group is
$0$. For example, with $(p,q)=(7,29)$, $(13,23)$, $(17,19)$, one of $p$ or $q$ is congruent
to $1 \pmod{4}$ and $\left( \dfrac{p}{q} \right)=\left( \dfrac{q}{p} \right)=1$,
so the homogeneous space $C_{p}$ is
everywhere locally solvable by Lemma~\ref{lem:T2}.
Hence the dimension of the $2$-Selmer group is positive.
However, using the \verb!ellrank()! function in PARI/GP, one finds that the
rank of the Mordell-Weil group, $E_{p,q}(\bbQ)$, is $0$ for these cases. The
reader may be interested to know that, subject to BSD, we found that $|\Sha|=4$
for each of these three examples, again using PARI/GP.
\end{example}

We are also able to obtain the following result for the rank over $\bbQ(i)$.

\begin{corollary}
\label{cor:rk0-Qi}
If any of the conditions in Theorem~$\ref{thm:rk0}$ holds, then the rank of
$E_{p,q}(\bbQ(i))$ is $0$.
\end{corollary}

\section{Descent via $2$-isogeny}
\label{sect:2-iso}

Our proof will use the method
of descent via $2$-isogeny. We will describe it briefly here. Our description is
based on Proposition~X.4.9 in \cite{S09}. For more details and information, as
well as proofs, see Section~X.4 of \cite{S09}. Suppose that
$E: Y^{2}= X^{3} + aX^{2} + bX$ is an elliptic curve over $\bbQ$ with $a,b \in \bbZ$
and $\widehat{E}: Y^{2} = X^{3} - 2aX^{2} + \left( a^{2}- 4b \right)X$
is the corresponding isogenous curve to $E$. Hence, there exists an isogeny
$\phi: E \rightarrow \widehat{E}$ of degree $2$ given by
\[
\phi(x,y) = \left( \frac{y^{2}}{x^{2}}, \frac{y(x^{2}-b)}{x^{2}} \right)
\]
for $x \neq 0$, with $\phi(0,0)=\phi(O)=O$.

Let $M_{\bbQ}$ be the set of all inequivalent normalised absolute values on $\bbQ$.
It consists of the ordinary absolute value, along with one non-archimedean absolute
value for each prime number $\ell$. As in \cite{S09}, we let $M_{\bbQ}^{0}$ be
the set of these non-archimedean absolute values. We also let
\[
S= \left\{ \infty \right\} \cup \left\{ v \in M_{\bbQ}^{0}: \ord_{v} \left( 2b \left( a^{2}-4b \right) \right)>0
\right\}.
\]

As in Theorem~X.1.1(c) of \cite{S09} with $m=2$ and $K=\bbQ$ there, we put
\[
\bbQ(S,2) = \left\{ d \in \bbQ^{*}/ \left( \bbQ^{*} \right)^{2}: \ord_{v}(d) \equiv 0 \, (\myMod 2)
\text{ for all $v \not\in S$} \right\}.
\]

For each $d \in \bbQ^{*}$, we let $C_{d}/\bbQ$ be the homogeneous space for $E/\bbQ$
given by the equation
\[
C_{d}: dw_{1}^{2}=d^{2}-2adz^{2}+ \left( a^{2}-4b \right) z^{4}.
\]

In fact, we will work primarily with an equation obtained by homogenising this equation
and then replacing $uw_{1}$ on the left-hand side by $w$:
\[
dw^{2}=d^{2}u^{4}-2ad(uz)^{2}+ \left( a^{2}-4b \right) z^{4}.
\]

Dividing this equation by $d$, we obtain
\[
w^{2}=du^{4}-2a(uz)^{2}+ \frac{a^{2}-4b}{d} z^{4}.
\]

The $\phi$-Selmer group is
\[
S^{(\phi)}(E/\bbQ) = \left\{ d \in \bbQ(S,2): C_{d} \left( \bbQ_{v} \right) \neq \emptyset
\text{ for all $v \in S$} \right\}.
\]

Here $C_{d} \left( \bbQ_{v} \right) \neq \emptyset$ is equivalent to the equation
$w^{2}=du^{4}-2a(uz)^{2}+ \frac{a^{2}-4b}{d} z^{4}$ having a solution in $\bbQ_{v}$
in addition to the trivial solution. That is, it has a solution in $\bbQ_{v}$
other than $(u,w,z)=(0,0,0)$.

As Silverman states in Remark~X.4.9.1 in \cite{S09}, since $\widehat{E}$ is of
the same form as $E$, everything above also applies to $\widehat{E}$ with the
dual isogeny $\widehat{\phi}: \widehat{E} \rightarrow E$. Here
\[
\widehat{S}
=\left\{ \infty \right\} \cup \left\{ v \in M_{\bbQ}^{0}: \ord_{v} \left( 32b \left( a^{2}-4b \right) \right)>0 \right\}=S.
\]

Thus,
\[
\bbQ \left( \widehat{S},2 \right)
= \left\{ d \in \bbQ^{*}/ \left( \bbQ^{*} \right)^{2}: \ord_{v}(d) \equiv 0 \, (\myMod 2)
\text{ for all $v \not\in \widehat{S}$} \right\}.
\]

For each $d \in \bbQ^{*}$, we let $\widehat{C}_{d}/\bbQ$ be the homogeneous
space for $\widehat{E}/\bbQ$ given by the equation
\[
\widehat{C}_{d}: dw_{1}^{2}=d^{2}+4adz^{2} + 16bz^{4}
\]
and the $\widehat{\phi}$-Selmer group is
\[
S^{(\widehat{\phi})} \left( \widehat{E}/\bbQ \right)
= \left\{ d \in \bbQ \left( \widehat{S},2 \right):
\widehat{C}_{d} \left( \bbQ_{v} \right) \neq \emptyset
\text{ for all $v \in \widehat{S}$} \right\}.
\]

As with $C_{d}$, we will use a version of the equation for $\widehat{C}_{d}$
obtained via homogenisation in what follows:
\[
dw^{2}=d^{2}u^{4}+4ad(uz)^{2}+16bz^{4},
\]
or
\[
w^{2}=du^{4}+4a(uz)^{2}+\frac{16b}{d}z^{4}.
\]

The following result is crucial for obtaining information about both the Mordell-Weil
rank and the $2$-torsion subgroup of the Tate-Shafarevich group. It follows from
results in Chapter~X of \cite{S09}, but we have not been able to find a statement
in a similar form in the literature.

The isogeny $\phi: E \rightarrow \widehat{E}$ induces a map $\phi_{S}:
S^{(2)}(E/\bbQ) \rightarrow S^{(\widehat{\phi})} \left( \widehat{E}/\bbQ \right)$.
We let $\cC_{S}=\coker \left( \phi_{S} \right)$.

\begin{proposition}
\label{prop:2.1}
Let $E$, $\widehat{E}$, $\phi$ and $\widehat{\phi}$ be as defined earlier in this
section. Then,
\[
\rank E(\bbQ)+\dim_{2} \Sha (E/\bbQ)[2]
= \dim_{2} S^{(\phi)}(E/\bbQ)+\dim_{2} S^{(\widehat{\phi})}(\widehat{E}/\bbQ)-2
- \dim_{2} \cC_{S}.
\]
\end{proposition}

For any abelian group $A$, we write $A[m]$ for the $m$-torsion subgroup of $A$,
that is, the subgroup of elements, $a \in A$ such that $ma$ is the
identity element of $A$.

\begin{proof}
We start with the exact sequence in Theorem~X.4.2(a) of \cite{S09} applied to
the isogenies $\phi$ and $\widehat{\phi}$:
\begin{align}
\label{eq:exact1}
& 0 \rightarrow \widehat{E}(\bbQ)/\phi \left( E(\bbQ) \right) \rightarrow
S^{(\phi)}(E/\bbQ) \rightarrow \Sha(E/\bbQ)[\phi] \rightarrow 0 \\
& 0 \rightarrow E(\bbQ)/\widehat{\phi} \left( \widehat{E}(\bbQ) \right) \rightarrow
S^{(\widehat{\phi})} \left( \widehat{E}/\bbQ \right) \rightarrow
\Sha \left( \widehat{E}/\bbQ \right) \left[ \widehat{\phi} \right] \rightarrow 0.
\end{align}

Considering these objects as $\bbF_{2}$-vector spaces, from these two exact
sequences, we obtain
\begin{align}
\label{eq:dim1}
\dim_{2} S^{(\phi)}(E/\bbQ)
& =\dim_{2} \widehat{E}(\bbQ)/\phi \left( E(\bbQ) \right) + \dim_{2} \Sha(E/\bbQ)[\phi] \\
\label{eq:dim2}
\dim_{2} S^{(\widehat{\phi})} \left( \widehat{E}/\bbQ \right)
&= \dim_{2} E(\bbQ)/\widehat{\phi} \left( \widehat{E}(\bbQ) \right)
+ \dim_{2} \Sha \left( \widehat{E}/\bbQ \right) \left[ \widehat{\phi} \right].
\end{align}

Now apply the exact sequence near the top of page~336 of \cite{S09} with $m=2$,
the degree of the isogeny $\phi$:
\[
0 \rightarrow \frac{\widehat{E}(\bbQ) \left[ \widehat{\phi} \right]}{\phi \left( E(\bbQ)[2] \right)}
\rightarrow \frac{\widehat{E}(\bbQ)}{\phi \left( E(\bbQ) \right)}
\stackrel{\widehat{\phi}}{\longrightarrow} \frac{E(\bbQ)}{2E(\bbQ)}
\rightarrow \frac{E(\bbQ)}{\widehat{\phi} \left( \widehat{E}(\bbQ) \right)} \rightarrow 0.
\]

Taking dimensions here, we obtain
\begin{equation}
\label{eq:dim3}
\dim_{2} \frac{\widehat{E}(\bbQ) \left[ \widehat{\phi} \right]}{\phi \left( E(\bbQ)[2] \right)}
+ \dim_{2} \frac{E(\bbQ)}{2E(\bbQ)}
= \dim_{2}\frac{\widehat{E}(\bbQ)}{\phi \left( E(\bbQ) \right)}
+ \dim_{2} \frac{E(\bbQ)}{\widehat{\phi} \left( \widehat{E}(\bbQ) \right)}.
\end{equation}

We have $\dim_{2} E(\bbQ)/2E(\bbQ)=\rank E(\bbQ) + \dim_{2} E(\bbQ)[2]$.

We also have $\dim_{2} \widehat{E}(\bbQ) \left[ \widehat{\phi} \right]=1$, since
$\ker \left( \widehat{\phi} \right)= \{ O,(0,0) \}$.

The restriction $\phi: E(\bbQ)[2] \rightarrow \widehat{E}(\bbQ) \left[ \widehat{\phi} \right]$
has kernel $E(\bbQ)[\phi]$, which has dimension $1$, so
$\dim_{2} \phi \left( E(\bbQ)[2] \right)=\dim_{2} E(\bbQ)[2]-1$. Thus
$\dim_{2} \widehat{E}(\bbQ) \left[ \widehat{\phi} \right]/\phi \left( E(\bbQ)[2] \right)
=1-\left( \dim_{2} E(\bbQ)[2]-1 \right)=2-\dim_{2} E(\bbQ)[2]$. So the left-hand
side of \eqref{eq:dim3} is
$2-\dim_{2} E(\bbQ)[2]+\rank E(\bbQ) + \dim_{2} E(\bbQ)[2]=\rank E(\bbQ)+2$. That is,
\[
\dim_{2} \frac{\widehat{E}(\bbQ)}{\phi \left( E(\bbQ) \right)}
+ \dim_{2} \frac{E(\bbQ)}{\widehat{\phi} \left( \widehat{E}(\bbQ) \right)}
=\rank E(\bbQ)+2.
\]

Adding \eqref{eq:dim1} and \eqref{eq:dim2} together and applying this relationship,
we obtain
\begin{equation}
\label{eq:dim4}
\dim_{2} S^{(\phi)}(E/\bbQ) + \dim_{2} S^{(\widehat{\phi})} \left( \widehat{E}/\bbQ \right)
=\rank E(\bbQ)+2 + \dim_{2} \Sha(E/\bbQ)[\phi]
+ \dim_{2} \Sha \left( \widehat{E}/\bbQ \right) \left[ \widehat{\phi} \right].
\end{equation}

We now consider the $\Sha$ terms there and relate them to $\Sha(E/\bbQ)[2]$.

The isogeny $\phi: E \rightarrow \widehat{E}$ induces a map $\phi_{\Sha}:
\Sha (E/\bbQ) \rightarrow \Sha \left( \widehat{E}/\bbQ \right)$. Since
$\widehat{\phi} \circ \phi=[2]$, we have
\[
\phi_{\Sha} \left( \Sha(E/\bbQ)[2] \right)
\subseteq \Sha \left( \widehat{E}/\bbQ \right) \left[ \widehat{\phi} \right].
\]

Also
\[
\ker \left( \phi_{\Sha}: \Sha(E/\bbQ)[2]
\rightarrow \Sha \left( \widehat{E}/\bbQ \right) \left[ \widehat{\phi} \right] \right)
= \left\{ \xi \in \Sha(E/\bbQ)[2]: \phi_{\Sha}(\xi)=0 \right\},
\]
$\Sha (E/\bbQ)[\phi]= \left\{ \xi \in \Sha(E/\bbQ): \phi_{\Sha}(\xi)=0 \right\}$ and
$\Sha (E/\bbQ)[\phi] \subseteq \Sha (E/\bbQ)[2]$ (using $\widehat{\phi} \circ \phi=[2]$).
So
\[
\ker \left( \phi_{\Sha}: \Sha(E/\bbQ)[2]
\rightarrow \Sha \left( \widehat{E}/\bbQ \right) \left[ \widehat{\phi} \right] \right)
= \Sha (E/\bbQ)[\phi].
\]

Hence there is an exact sequence
\[
0 \rightarrow \Sha(E/\bbQ)[\phi] \rightarrow \Sha (E/\bbQ)[2]
\stackrel{\phi_{\Sha}}{\longrightarrow}
\Sha \left( \widehat{E}/\bbQ \right) \left[ \widehat{\phi} \right]
\rightarrow \cC_{\Sha} \rightarrow 0,
\]
where $\cC_{\Sha}=\coker \left( \phi_{\Sha}: \Sha(E/\bbQ)[2] \rightarrow \Sha \left( \widehat{E}/\bbQ \right) \left[ \widehat{\phi} \right] \right)$.

Taking dimensions in that exact sequence, we have
\[
\dim_{2} \Sha(E/\bbQ)[\phi]
+\dim_{2} \Sha \left( \widehat{E}/\bbQ \right) \left[ \widehat{\phi} \right]
= \dim_{2} \Sha (E/\bbQ)[2] + \dim_{2} \cC_{\Sha}.
\]

Applying this to \eqref{eq:dim4} above, we have
\begin{equation}
\label{eq:dim5}
\dim_{2} S^{(\phi)}(E/\bbQ) + \dim_{2} S^{(\widehat{\phi})} \left( \widehat{E}/\bbQ \right)
=\rank E(\bbQ)+2 + \dim_{2} \Sha (E/\bbQ)[2] + \dim_{2} \cC_{\Sha}.
\end{equation}

However, it will be easier for us to work with $\cC_{S}$ and its elements than
with $\cC_{\Sha}$. We now show that these two cokernels are isomorphic.

From Theorem~X.4.2(a) in \cite{S09}, we have
\[
0 \rightarrow E(\bbQ)/2E(\bbQ) \rightarrow
S^{(2)}(E/\bbQ) \rightarrow \Sha(E/\bbQ)[2] \rightarrow 0
\]
and
\begin{equation}
\label{eq:exact2}
0 \rightarrow E(\bbQ)/\widehat{\phi} \left( \widehat{E}(\bbQ) \right) \rightarrow
S^{(\widehat{\phi})} \left( \widehat{E}/\bbQ \right) \rightarrow
\Sha \left( \widehat{E}/\bbQ \right) \left[ \widehat{\phi} \right] \rightarrow 0.
\end{equation}

We make this into the following larger commutative diagram:
\begin{equation}
\label{eq:commutative}
\begin{tikzcd}
0 \arrow[r] & E(\bbQ)/2E(\bbQ) \arrow[d, "f"] \arrow{r} & S^{(2)}(E/\bbQ)
  \arrow[d, "\phi_{S}"] \arrow{r} & \Sha(E/\bbQ)[2] \arrow[d, "\phi_{\Sha}"] \arrow{r} & 0 \\
0 \arrow{r} & E(\bbQ)/\widehat{\phi} \left( \widehat{E}(\bbQ) \right) \arrow{r}
& S^{(\widehat{\phi})} \left( \widehat{E}/\bbQ \right) \arrow{r}
& \Sha \left( \widehat{E}/\bbQ \right) \left[ \widehat{\phi} \right] \arrow{r} & 0.
\end{tikzcd}
\end{equation}

The left vertical map $f$ is the natural quotient map
$E(\bbQ)/2E(\bbQ) \rightarrow
E(\bbQ)/\widehat{\phi} \left( \widehat{E}(\bbQ) \right)$. It is surjective because
$2E(\bbQ)=\widehat{\phi} \circ \phi \left( E(\bbQ) \right)
\subseteq \widehat{\phi} \left( \widehat{E}(\bbQ) \right)$. Since it is surjective,
its cokernel is $0$.

The snake lemma tells us that $0 \rightarrow \ker(f) \rightarrow \ker \left( \phi_{S} \right)
\rightarrow \ker \left( \phi_{\Sha} \right) \rightarrow \coker(f) \rightarrow \coker \left( \phi_{S} \right)
\rightarrow \coker \left( \phi_{\Sha} \right) \rightarrow 0$. Since $\coker(f)=0$,
we have $0 \rightarrow \coker \left( \phi_{S} \right) \rightarrow \coker \left( \phi_{\Sha} \right) \rightarrow 0$.
Hence $\cC_{S}=\coker \left( \phi_{S} \right) \cong \coker \left( \phi_{\Sha} \right)=\cC_{\Sha}$,
as stated. The proposition now follows from this and \eqref{eq:dim5}.
\end{proof}

\begin{lemma}
\label{lem:Cphi-1}
The order of $\cC_{S}$ is a square.
\end{lemma}

\begin{proof}
This follows from Theorem~2.1 of \cite{Fi} with $p=2$. His $S_{1}$ and $S_{1}'$
are the same as our $S^{(\phi)} \left( E/\bbQ \right)$ and
$S^{(\widehat{\phi})} \left( \widehat{E}/\bbQ \right)$, respectively.

For $m=2$, Fisher defines $S_{2}'$ to be the image of
$S^{(2)}(E/\bbQ)$ in $S_{1}'$ under the map induced by $\phi$. Applying his
Theorem~2.1(i) with $m=1$ and using his notation, there are alternating pairings
\[
\Theta_{1}: S_{1} \times S_{1} \rightarrow \bbF_{p}
\text{ and }
\Theta_{1}': S_{1}' \times S_{1}' \rightarrow \bbF_{p},
\]
with kernels $S_{2}$ and $S_{2}'$. Hence the pairing $\Theta_{1}'$ gives a
nondegenerate alternating pairing
on $S_{1}'/S_{2}' \times S_{1}'/S_{2}'$, where
$S_{1}'/S_{2}'=S^{(\widehat{\phi})} \left( \widehat{E}/\bbQ \right)/ \Image \left( S^{(2)}(E/\bbQ) \rightarrow S^{(\widehat{\phi})} \left( \widehat{E}/\bbQ \right) \right)$.
This quotient is our $\cC_{S}$, defined above.

From linear algebra, we know that a finite-dimensional vector space with a
nondegenerate alternating bilinear pairing has even dimension.
\end{proof}

\begin{lemma}
\label{lem:Cphi-2}
{\rm (a)} Put $\dim_{2} S^{(\widehat{\phi})} \left( \widehat{E}/\bbQ \right)=s$
and $\dim_{2} E(\bbQ)/\widehat{\phi} \left( \widehat{E}(\bbQ) \right)=r$. We have
$\dim_{2} \cC_{S}$ is an even integer satisfying $0 \leq \dim_{2} \cC_{S} \leq s-r$.

\noindent
{\rm (b)} If $\dim_{2} S^{(\widehat{\phi})} \left( \widehat{E}/\bbQ \right)=1$,
then $\dim_{2} \cC_{S}=0$.

\noindent
{\rm (c)} If $\dim_{2} S^{(\widehat{\phi})} \left( \widehat{E}/\bbQ \right)=2$ and
$\widehat{E}(\bbQ)[2] \cong \bbZ/2\bbZ$, then
$E(\bbQ)/\widehat{\phi} \left( \widehat{E}(\bbQ) \right) \neq \{ 0 \}$ and so
$\dim_{2} \cC_{S}=0$.
\end{lemma}

\begin{proof}
(a) From \eqref{eq:exact2}, we know that the group
$E(\bbQ)/\widehat{\phi} \left( \widehat{E}(\bbQ) \right)$ injects into
$S^{(\widehat{\phi})} \left( \widehat{E}/\bbQ \right)$. Hence by our assumption
that $\dim_{2} E(\bbQ)/\widehat{\phi} \left( \widehat{E}(\bbQ) \right)=r$, it
follows that the dimension of this image is also $r$.

We showed after \eqref{eq:commutative} that the map $f$ from $E(\bbQ)/2E(\bbQ)$ into
$E(\bbQ)/\widehat{\phi} \left( \widehat{E}(\bbQ) \right)$ is surjective, so the
image of $E(\bbQ)/\widehat{\phi} \left( \widehat{E}(\bbQ) \right)$ in
$S^{(\widehat{\phi})} \left( \widehat{E}/\bbQ \right)$
is contained in $\Image \left( \phi_{S} \right)$.
Hence $\dim_{2} \Image \left( \phi_{S} \right) \geq r$.

Thus
$\dim_{2} \coker \left( \phi_{S} \right)
=\dim_{2} \dfrac{S^{(\widehat{\phi})} \left( \widehat{E}/\bbQ \right)}{\Image \left( \phi_{S} \right)}
\leq s-r$.
By Lemma~\ref{lem:Cphi-1}, it follows that $\dim_{2} \coker \left( \phi_{S} \right)$
is even.

\vspace*{2.0mm}

(b) This is an immediate consequence of part~(a).

\vspace*{2.0mm}

(c) Consider $T=(0,0) \in E(\bbQ)[2]$. Suppose that $T=\widehat{\phi} \left( Q \right)$
for some $Q \in \widehat{E}(\bbQ)$. Then $[2]Q=\phi \circ \widehat{\phi}(Q)=\phi(T)=O$.
So $Q \in \widehat{E}(\bbQ)[2]$.
However, we know that $\widehat{E} \left[ \widehat{\phi} \right]$
is a rational subgroup of order $2$ since $\widehat{\phi}$ is a rational $2$-isogeny.
So $\widehat{E} \left[ \widehat{\phi} \right] \subseteq \widehat{E}[2]$. Since
$\widehat{T}=(0,0)$ generates $\widehat{E} \left[ \widehat{\phi} \right]$, we
actually have $\widehat{E} \left[ \widehat{\phi} \right] \subseteq \widehat{E}(\bbQ)[2]$.
Here $\widehat{E}(\bbQ)[2] \cong \bbZ/2\bbZ$, so
$\widehat{E} \left[ \widehat{\phi} \right]=\widehat{E}(\bbQ)[2] \cong \bbZ/2\bbZ$.

From that we know that $\widehat{\phi}(Q)=O$. This contradicts $\widehat{\phi}(Q)=T$.
So $T \not\in \widehat{\phi} \left( \widehat{E}(\bbQ) \right)$ and the class of
$T$ gives a nonzero element of $E(\bbQ)/\widehat{\phi} \left( \widehat{E}(\bbQ) \right)$.

The remainder of the result follows from part~(a).
%
\end{proof}

\section{Homogeneous Spaces}
\label{sect:homo}

We will use the following result to establish conditions for the local solvability
of the homogeneous spaces that arise in our work.

\begin{lemma}
\label{lem:Cohen}
Let $a$, $b$ and $c$ be nonzero integers such that $a$ and $b$ are fourth-power-free,
$c$ is squarefree and $\gcd(a,b,c) = 1$.

\begin{enumerate}
\item The equation $ax^{4} + by^{4} + cz^{2} = 0$ has a nontrivial solution in
every $\bbQ_{p}$ for which $p \nmid (2abc)$. It has a nontrivial solution in $\bbR$
if and only if $a$, $b$ and $c$ do not all have the same sign.

\item Let $p | (abc)$ with $p \neq 2$. Reorder $a$ and $b$ so that
$v_{p}(a) \leq v_{p}(b)$ and set $\bv=\left( v_{p}(a), v_{p}(b), v_{p}(c) \right)$.

\begin{enumerate}
\item If $\bv=(2, 2, 0)$, then the equation has a nontrivial solution in $\bbQ_{p}$.

\item If $\bv = (0, 2, 1)$ or $\bv = (1, 3, 0)$, then the
equation does not have any nontrivial solutions in $\bbQ_{p}$.

\item If $\bv = (0, 0, 1)$, $\bv = (1, 1, 0)$ or $\bv = (3, 3, 0)$,
then the equation has a nontrivial solution in $\bbQ_{p}$
if and only if $-a/b$ is a fourth power in $\bbF_{p}^{*}$.

\item If $\bv = (0, 2, 0)$, then the equation has a nontrivial
solution in $\bbQ_{p}$ if and only if either $-a/c$ or
$-b/ \left( p^{2}c \right)$ is a square in $\bbF_{p}^{*}$.

\item Otherwise, if $\bv = (0, 1, 0)$ or $\bv = (0, 3, 0)$,
set $\alpha = -a/c$, if $\bv = (0, 1, 1)$ set $\alpha = -b/c$,
if $\bv = (0, 3, 1)$ or $\bv = (1, 2, 0)$ set
$\alpha = -b/\left( p^{2}c \right)$, and if $\bv = (2, 3, 0)$
set $\alpha = -a/ \left( p^{2}c \right)$. The equation has a
nontrivial solution in $\bbQ_{p}$ if and only if $\alpha$
is a square in $\bbF_{p}^{*}$.
\end{enumerate}

\item Assume that $p = 2$. Reorder $a$ and $b$ so that $v_{p}(a) \leq v_{p}(b)$
and set $\bw = \left( v_{p}(a), v_{p}(b), v_{p}(c) \right)$. For $\bw = (0, 3, 0)$,
$(0, 3, 1)$, $(1, 1, 0)$, $(1, 2, 0)$, $(2, 2, 0)$, $(2, 3, 0)$ or $(3, 3, 0)$,
replace $(a, b, c)$ by $(b/8, 2a, 2c)$, $(b/8, 2a, c/2)$, $(a/2, b/2, 2c)$,
$(a/2, b/2, 2c)$, $(a/4, b/4, c)$, $(a/4, b/4, c)$ or $(a/8, b/8, 2c)$ respectively.
Otherwise keep $a$, $b$ and $c$ unchanged. Finally, set
$\bv = \left( v_{p}(a), v_{p}(b), v_{p}(c) \right)$.
\begin{enumerate}
\item If $\bv = (0, 2, 1)$ or $\bv = (1, 3, 0)$, then the equation does not
have any nontrivial solutions in $\bbQ_{2}$.

\item If $\bv = (0,0,0)$, then the equation has a nontrivial solution in $\bbQ_{2}$
if and only if $8 | (a + c)$, $8 | (b + c)$, $16 | (a + b)$ or $16 | (a + b + 4c)$.

\item If $\bv = (0, 0, 1)$, then the equation has a nontrivial solution in $\bbQ_{2}$
if and only if $8 | (a + b)$ or $16 | (a + b + c)$.

\item If $\bv = (0, 1, 0)$, then the equation has a nontrivial solution in $\bbQ_{2}$
if and only if $8 | (a + c)$ or $8 | (a + b + c)$.

\item If $\bv = (0, 2, 0)$, then the equation has a nontrivial solution in $\bbQ_{2}$
if and only if $8 | (a + c)$, $8 | (a + b + c)$ or $16 | (b + 4c)$.

\item If $\bv = (0, 1, 1)$, then the equation has a nontrivial solution in $\bbQ_{2}$
if and only if $16 | (b + c)$.
\end{enumerate}
\end{enumerate}
\end{lemma}

\begin{proof}
This is Proposition~6.5.1 in \cite[pp. 389--390]{Co}.
\end{proof}

\subsection{Homogeneous spaces for $E_{p,q}$}
\label{subsect:E-hom-spaces}

In the notation of the previous section, we have $a=0$ and $b=-pq$. So
$S= \{ \infty, 2, p, q \}$. Hence
\[
\bbQ \left( S,2 \right) = \{ \pm 1, \pm 2, \pm p, \pm q, \pm 2p, \pm 2q, \pm pq, \pm 2pq \}.
\]

$d=1$ and $d=pq$ are always elements of $S^{(\phi)} \left( E_{p,q}/\bbQ \right)$,
so we need not consider the local solvability of their associated homogeneous
spaces here.

If $d<0$, then the left-hand side of $C_{d}: dw^{2}=d^{2}u^{4}+4pqz^{4}$ is non-positive
while the right-hand side is strictly positive. Thus, the corresponding homogeneous
spaces are not solvable over $\bbR$.

This leaves $6$ values of $d$ to be considered: $\{ 2, p, q, 2p, 2q, 2pq \}$.

But notice that only three of the equations $C_{d}$ are actually distinct modulo
squares. The equation for $d=2$ is the same as the equation for $d=2pq$ modulo
squares. Similarly, the equations for $d=p$ and $d=q$ are the same,
as are the equations for $d=2p$ and $d=2q$, modulo squares.

We then consider the solvability of the following homogeneous spaces over local
fields.
\begin{center}
\begin{table}[h]
\begin{tabular}{|c|c|l|}\hline
\multicolumn{1}{|c|}{Label} & \multicolumn{1}{|c|}{$d$} & \multicolumn{1}{|c|}{Equation} \\ \hline
$C_{2}$   &    $2$  & $w^{2} =  2u^{4}+2pqz^{4}$ \\[1pt] \hline
$C_{p}$   &    $p$  & $w^{2} =  pu^{4}+ 4qz^{4}$ \\[1pt] \hline
$C_{2p}$  &   $2p$  & $w^{2} = 2pu^{4}+ 2qz^{4}$ \\[1pt] \hline
\end{tabular}
\caption{Homogeneous spaces for $E_{p,q}$}
\label{table:H-homogeneous spaces}
\end{table}
\end{center}

\begin{lemma}
\label{lem:T1}
The homogeneous space defined by $C_{2}: w^{2}=2u^{4}+2pqz^{4}$ has nontrivial
local solutions in $\bbR$ and in $\bbQ_{\ell}$ for all primes $\ell$ if and
only if $p, q \equiv \pm 1 \pmod{8}$ and either $pq \equiv 7 \pmod{8}$ or
$pq \equiv 1 \pmod{16}$.
\end{lemma}

\begin{proof}
Here, and in our proofs for all the homogeneous-space lemmas, we will use
Lemma~\ref{lem:Cohen} to establish the lemma.

We apply Lemma~\ref{lem:Cohen} with $a=2$, $b=2pq$ and $c=-1$. By
part~(1) of that lemma, since $a$, $b$ and $c$ do not all have the same
sign, there is always a nontrivial solution in $\bbR$.

Also, by part~(1) of Lemma~\ref{lem:Cohen}, there is a nontrivial solution in
$\bbQ_{\ell}$ for all $\ell \nmid (2pq)$.

If $\ell$ is either $p$ or $q$, then $v_{\ell}(a)=0 \leq v_{\ell}(b)=1$ and
$\bv=(0,1,0)$ in the notation of Lemma~\ref{lem:Cohen}. So part~(2-e) of
Lemma~\ref{lem:Cohen}
applies. $\alpha=-a/c=2$, which is a square in $\bbF_{\ell}^{*}$ if and only if
$\ell \equiv \pm 1 \pmod{8}$. Hence there is a local solution in $\bbQ_{p}$ and in
$\bbQ_{q}$ if and only if $p,q \equiv \pm 1 \pmod{8}$.

Lastly, we consider $\ell=2$. Here $\bw=(1,1,0)$, so by part~(3) of Lemma~\ref{lem:Cohen},
we replace $(a,b,c)$ with $(1,pq,-2)$ and put $\bv=(0,0,1)$. This means we are
in part~(3-c) of Lemma~\ref{lem:Cohen}. Since $a+b=pq+1$ and $a+b+c=pq-1$,
there is a nontrivial solution in $\bbQ_{2}$ if and only if $8|(pq+1)$ or
$16|(pq-1)$.
\end{proof}

\begin{lemma}
\label{lem:T2}
The homogeneous space $C_{p}: w^{2} =  pu^{4}+ 4qz^{4}$ has nontrivial local
solutions in $\bbR$ and in $\bbQ_{\ell}$ for all primes $\ell$ if and only if
$\left( \dfrac{p}{q} \right)=1$ and $\left( \dfrac{q}{p} \right)=1$ both hold
and at least one of $p \equiv 1 \pmod{4}$ or $q \equiv 1 \pmod{4}$ holds.
\end{lemma}

\begin{proof}
We apply Lemma~\ref{lem:Cohen} here with $a=p$, $b=4q$ and $c=-1$. By
part~(1) of that lemma, since $a$, $b$ and $c$ do not all have the same
sign, there is always a nontrivial solution in $\bbR$.

Also, by part~(1) of Lemma~\ref{lem:Cohen}, there is a nontrivial solution in
$\bbQ_{\ell}$ for all $\ell \nmid (2pq)$.

For $q$, we have $v_{q}(a)=0 \leq v_{q}(b)=1$ and $\bv=(0,1,0)$. Part~(2-e) of
Lemma~\ref{lem:Cohen} applies. We put $\alpha=-a/c=p$ and we find that there is
a local solution in $\bbQ_{q}$ if and only if $\left( \dfrac{p}{q} \right)=1$.

To consider $p$, we must reorder the terms in our equation, putting $a=4q$, $b=p$
and $c=-1$. So $v_{p}(a)=0 \leq v_{p}(b)=1$ and $\bv=(0,1,0)$ again, as with $q$.
From part~(2-e) of Lemma~\ref{lem:Cohen}, $\alpha=4q$, so there is a local solution
in $\bbQ_{p}$ if and only if $\left( \dfrac{q}{p} \right)=1$.

To consider $\ell=2$, we have $\bv=\bw=(0,2,0)$, using our initial ordering of
the terms. Part~(3-e) of Lemma~\ref{lem:Cohen} applies. There is a nontrivial
solution in $\bbQ_{2}$ if and only if $8|(p-1)$,
$8|(p+4q-1)$ or $16|(4q-4)$.
The condition $8|(p+4q-1)$ holds if and only if
$p \equiv 5 \pmod{8}$. Combining this with $p \equiv 1 \pmod{8}$ (from $8|(p-1)$)
and $q \equiv 1 \pmod{4}$ (from $16|(4q-4)$), the lemma follows.
\end{proof}

\begin{lemma}
\label{lem:T3}
The homogeneous space $C_{2p}: w^{2} = 2pu^{4}+ 2qz^{4}$ has nontrivial local
solutions in $\bbR$ and in $\bbQ_{\ell}$ for all primes $\ell$ if and only if
$\left( \dfrac{2p}{q} \right)=1$ and
$\left( \dfrac{2q}{p} \right)=1$ both hold and at least one of
$p+q \equiv 0 \pmod{8}$ or $p+q \equiv 2 \pmod{16}$ holds.
\end{lemma}

\begin{proof}
We apply Lemma~\ref{lem:Cohen} here with $a=2p$, $b=2q$ and $c=-1$. By
part~(1) of that lemma, since $a$, $b$ and $c$ do not all have the same
sign, there is always a nontrivial solution in $\bbR$.

Also, by part~(1) of Lemma~\ref{lem:Cohen}, there is a nontrivial solution in
$\bbQ_{\ell}$ for all $\ell \nmid (2pq)$.

For $q$, we have $v_{q}(a)=0 \leq v_{q}(b)=1$ and $\bv=(0,1,0)$. Part~(2-e) of
Lemma~\ref{lem:Cohen} applies. We put $\alpha=-a/c=2p$ and we find that there is
a local solution in $\bbQ_{q}$ if and only if $\left( \dfrac{2p}{q} \right)=1$.

To consider $p$, we must reorder the terms in our equation, putting $a=2q$, $b=2p$
and $c=-1$. So $v_{p}(a)=0 \leq v_{p}(b)=1$ and $\bv=(0,1,0)$ again, as with $q$.
From part~(2-e) of Lemma~\ref{lem:Cohen}, $\alpha=-a/c=2q$, so there is a local
solution in $\bbQ_{p}$ if and only if $\left( \dfrac{2q}{p} \right)=1$.

Lastly, we consider $\ell=2$. Here $\bw=(1,1,0)$, so by part~(3) of Lemma~\ref{lem:Cohen},
we replace $(a,b,c)$ with $(p,q,-2)$ and put $\bv=(0,0,1)$. This means we are
in part~(3-c) of Lemma~\ref{lem:Cohen}. Since $a+b=p+q$ and $a+b+c=p+q-2$,
there is a nontrivial solution in $\bbQ_{2}$ if and only if $8|(p+q)$ or
$16|(p+q-2)$.
\end{proof}

\subsection{Homogeneous spaces for $\widehat{E_{p,q}}$}
\label{subsect:hatE-hom-spaces}

Now we determine $S^{(\widehat{\phi})} \left( \widehat{E_{p,q}}/\bbQ \right)$.
As above in Subsection~\ref{subsect:E-hom-spaces}, here we have
\[
\bbQ \left( \widehat{S}, 2 \right) = \{ \pm 1, \pm 2, \pm p, \pm q, \pm 2p, \pm 2q, \pm pq, \pm 2pq \}.
\]

Since $\widehat{E_{p,q}}: y^{2}= x^{3} + 4pqx$, we know that $1, -pq \in
S^{(\widehat{\phi})} \left( \widehat{E_{p,q}}/\bbQ \right)$.

As with $S^{(\phi)} \left( E_{p,q}/\bbQ \right)$, some of the spaces $\widehat{C}_{d}$
are equivalent to others after changing $d$ by a square and swapping the variables
$u$ and $z$. So
$\widehat{C}_{-1}$ is the same as $\widehat{C}_{pq}$;
$\widehat{C}_{2}$ is the same as $\widehat{C}_{-2pq}$;
$\widehat{C}_{-2}$ is the same as $\widehat{C}_{2pq}$;
$\widehat{C}_{p}$ is the same as $\widehat{C}_{-q}$;
$\widehat{C}_{-p}$ is the same as $\widehat{C}_{q}$;
$\widehat{C}_{2p}$ is the same as $\widehat{C}_{-2q}$; and
$\widehat{C}_{-2p}$ is the same as $\widehat{C}_{2q}$.

Hence we are left with the seven values of $d$ in Table~\ref{table:oH-homogeneous spaces}.

\begin{center}
\begin{table}[h]
\begin{tabular}{|c|r|l|}\hline
\multicolumn{1}{|c|}{Label} & \multicolumn{1}{|c|}{$d$} & \multicolumn{1}{|c|}{Equation} \\ \hline
$\widehat{C}_{-1}$   &   $-1$  & $w^{2} =   -u^{4}+  pqz^{4}$ \\[2pt] \hline
$\widehat{C}_{2}$    &    $2$  & $w^{2} =   2u^{4}- 8pqz^{4}$ \\[2pt] \hline
$\widehat{C}_{-2}$   &   $-2$  & $w^{2} =  -2u^{4}+ 8pqz^{4}$ \\[2pt] \hline
$\widehat{C}_{p}$    &    $p$  & $w^{2} =   pu^{4}-   qz^{4}$ \\[2pt] \hline
$\widehat{C}_{-p}$   &   $-p$  & $w^{2} =  -pu^{4}+   qz^{4}$ \\[2pt] \hline
$\widehat{C}_{2p}$   &   $2p$  & $w^{2} =  2pu^{4}-  8qz^{4}$ \\[2pt] \hline
$\widehat{C}_{-2p}$  &  $-2p$  & $w^{2} = -2pu^{4}+  8qz^{4}$ \\[2pt] \hline
\end{tabular}
\caption{Homogeneous spaces for $\widehat{E_{p,q}}$}
\label{table:oH-homogeneous spaces}
\end{table}
\end{center}

\begin{remark-nonum}
For $\widehat{C}_{-1}$, $\widehat{C}_{p}$ and $\widehat{C}_{-p}$, we have made
the change of variables $z \mapsto 2z$ to remove a factor of $16$ from the
coefficient of $z^{4}$.
\end{remark-nonum}

\begin{lemma}
\label{lem:oT1}
The homogeneous space $\widehat{C}_{-1}: w^{2}=-u^{4}+pqz^{4}$ has nontrivial
local solutions in $\bbR$ and in $\bbQ_{\ell}$ for all primes $\ell$ if and only
if $p \equiv q \equiv 1 \pmod{4}$ and $pq \equiv 1,5,9 \pmod{16}$.
\end{lemma}

\begin{proof}
We apply Lemma~\ref{lem:Cohen} with $a=-1$, $b=pq$ and $c=-1$. By
part~(1) of that lemma, since $a$, $b$ and $c$ do not all have the same
sign, there is always a nontrivial solution in $\bbR$.

Also, by part~(1) of Lemma~\ref{lem:Cohen}, there is a nontrivial solution in
$\bbQ_{\ell}$ for all $\ell \nmid (2pq)$.

If $\ell$ is either $p$ or $q$, then $v_{\ell}(a)=0 \leq v_{\ell}(b)=1$ and $\bv=(0,1,0)$
in the notation of Lemma~\ref{lem:Cohen}. So part~(2-e) of Lemma~\ref{lem:Cohen}
applies. $\alpha=-a/c=-1$, which is a square in $\bbF_{\ell}^{*}$ if and only if
$\ell \equiv 1 \pmod{4}$.

Lastly, we consider $\ell=2$. Here $\bv=\bw=(0,0,0)$. Part~(3-b) of
Lemma~\ref{lem:Cohen} applies. There is a nontrivial solution in
$\bbQ_{2}$ if and only if $8|(-2)$ (which never holds), $8|(pq-1)$,
$16|(pq-1)$ or $16|(pq-5)$. This condition holds if
and only if $pq \equiv 1,5,9 \pmod{16}$.
\end{proof}

\begin{lemma}
\label{lem:oT2}
The homogeneous spaces $\widehat{C}_{2}$, $\widehat{C}_{-2}$, $\widehat{C}_{2p}$
and $\widehat{C}_{-2p}$ have no nontrivial solutions in $\bbQ_{2}$.
\end{lemma}

\begin{proof}
For each of these homogeneous spaces, after reordering the $u$ and $z$ coefficients,
if necessary, $\bv=\bw=(1,3,0)$. By part~(3-a) of Lemma~\ref{lem:Cohen},
there are no nontrivial solutions in $\bbQ_{2}$.
\end{proof}

\begin{lemma}
\label{lem:oT3}
The homogeneous space $\widehat{C}_{p}: w^{2}=pu^{4}-qz^{4}$ has nontrivial
local solutions in $\bbR$ and in $\bbQ_{\ell}$ for all primes $\ell$ if and only
if $\left( \dfrac{-q}{p} \right)=1$ and $\left( \dfrac{p}{q} \right)=1$ both
hold and at least one of $p \equiv 1 \pmod{8}$, $q \equiv 7 \pmod{8}$,
$p-q \equiv 0 \pmod{16}$ or $p-q \equiv 4 \pmod{16}$ holds.
\end{lemma}

\begin{proof}
We apply Lemma~\ref{lem:Cohen} with $a=p$, $b=-q$ and $c=-1$. By
part~(1) of that lemma, since $a$, $b$ and $c$ do not all have the same
sign, there is always a nontrivial solution in $\bbR$.

Also, by part~(1) of Lemma~\ref{lem:Cohen}, there is a nontrivial solution in
$\bbQ_{\ell}$ for all $\ell \nmid (2pq)$.

If $\ell=q$, $v_{q}(a)=0 \leq v_{q}(b)=1$ and $\bv=(0,1,0)$. Part~(2-e) of
Lemma~\ref{lem:Cohen} applies. We put $\alpha=-a/c=p$ and we find that there is a local
solution in $\bbQ_{q}$ if and only if $\left( \dfrac{p}{q} \right)=1$.

To consider $p$, we must reorder the terms in our equation, putting $a=-q$, $b=p$
and $c=-1$. So $v_{p}(a)=0 \leq v_{p}(b)=1$ and $\bv=(0,1,0)$. Here $\alpha=-q$,
so there is a local solution in $\bbQ_{p}$ if and only if
$\left( \dfrac{-q}{p} \right)=1$.

Lastly, we consider $\ell=2$. Here $\bv=\bw=(0,0,0)$.
Part~(3-b) of Lemma~\ref{lem:Cohen} applies. There is a nontrivial solution in
$\bbQ_{2}$ if and only if $8|(p-1)$, $8|(-q-1)$, $16|(p-q)$ or $16|(p-q-4)$.
\end{proof}

\begin{lemma}
\label{lem:oT4}
The homogeneous space $\widehat{C}_{-p}: w^{2}=-pu^{4}+qz^{4}$ has nontrivial
local solutions in $\bbR$ and in $\bbQ_{\ell}$ for all primes $\ell$ if and only
if $\left( \dfrac{q}{p} \right)=1$ and $\left( \dfrac{-p}{q} \right)=1$ both
hold and at least one of $p \equiv 7 \pmod{8}$, $q \equiv 1 \pmod{8}$,
$q-p \equiv 0 \pmod{16}$ or $q-p \equiv 4 \pmod{16}$ holds.
\end{lemma}

\begin{proof}
This follows from Lemma~\ref{lem:oT3} after interchanging $p$ and $q$ and
swapping the variables $u$ and $z$.
\end{proof}

\section{Proof of Theorem~\ref{thm:rk0}}

We proceed by considering the conditions we obtained for the homogeneous spaces
in the previous section.

$\bullet$ Consider first the case when $p \equiv q \equiv 1 \pmod{4}$.

We consider the conditions on the homogeneous spaces in order of ease of application.

For $C_{p}$, the conditions $\left( \dfrac{p}{q} \right)=\left( \dfrac{q}{p} \right)=1$
in Lemma~\ref{lem:T2} are equivalent here due to quadratic reciprocity.
So we require that $\left( \dfrac{p}{q} \right)=-1$ in order that $C_{p}$ is not
everywhere locally solvable.

For $\widehat{C}_{-1}$, we require $pq \equiv 13 \pmod{16}$ so that $\widehat{C}_{-1}$
is not everywhere locally solvable by Lemma~\ref{lem:oT1}.

For $C_{2}$, $pq \equiv 7 \pmod{8}$ is never possible since $p \equiv q \equiv 1 \pmod{4}$.
So, by Lemma~\ref{lem:T1}, we want $pq \not\equiv 1 \pmod{16}$ or
$p \equiv 5 \pmod{8}$ or $q \equiv 5 \pmod{8}$ in order for
$C_{2}$ to not be everywhere locally solvable. The first of these conditions
holds by the condition we obtained for $\widehat{C}_{-1}$.

For $C_{2p}$, by Lemma~\ref{lem:T3}, we want one of $\left( \dfrac{2p}{q} \right)=-1$
or $\left( \dfrac{2q}{p} \right)=-1$ to hold or else neither $p+q \equiv 0 \pmod{8}$
nor $p+q \equiv 2 \pmod{16}$ holds.

From our consideration of $C_{p}$, we have $\left( \dfrac{p}{q} \right)=\left( \dfrac{q}{p} \right)=-1$.
So $\left( \dfrac{2q}{p} \right)=-\left( \dfrac{2}{p} \right)$ and
$\left( \dfrac{2p}{q} \right)=-\left( \dfrac{2}{q} \right)$. Hence
$\left( \dfrac{2p}{q} \right)=-1$ implies that $\left( \dfrac{2}{q} \right)=1$,
so $q \equiv 1 \pmod{8}$,
and similarly $\left( \dfrac{2q}{p} \right)=-1$ implies that $p \equiv 1 \pmod{8}$.

From $p \equiv q \equiv 1 \pmod{4}$ and $pq \equiv 13 \pmod{16}$, we have
$(p,q) \equiv (1,13), (5,9), (9,5), (13,1) \pmod{16}$. For each possibility,
one of $p$ or $q$ is congruent to $1 \pmod{8}$. Hence the corresponding value
of $\left( \dfrac{2p}{q} \right)$ or $\left( \dfrac{2q}{p} \right)$ is $-1$.
Therefore, under the conditions we established from considering $C_{p}$ and
$\widehat{C}_{-1}$, $C_{2p}$ is never everywhere locally solvable.

For $\widehat{C}_{p}$ to be everywhere locally solvable, we need
$\left( \dfrac{p}{q} \right)=1$ to hold by Lemma~\ref{lem:oT3}. But from our
consideration of $C_{p}$, this never holds and hence $\widehat{C}_{p}$ is never
everywhere locally solvable.

The same argument shows that $\widehat{C}_{-p}$ is also never everywhere locally
solvable.

In this way, we see that for $p \equiv q \equiv 1 \pmod{4}$, condition~(i) in
Theorem~\ref{thm:rk0} holds if and only if
$S^{(\phi)} \left( E_{p,q}/\bbQ \right) = \{1, pq \}$
and $S^{(\widehat{\phi})} \left( \widehat{E_{p,q}}/\bbQ \right) = \{1, -pq \}$
both hold.

\vspace*{2.0mm}

\noindent
$\bullet$ Suppose that $p \equiv 1 \pmod{4}$ and $q \equiv 3 \pmod{4}$.

$C_{2}$ is not everywhere locally solvable if and only if $p \equiv 5 \pmod{8}$
or $q \equiv 3 \pmod{8}$ or $pq \not\equiv 7 \pmod{8}$, by Lemma~\ref{lem:T1}.
Notice that $pq \equiv 1 \pmod{16}$ can never hold here.

For $C_{p}$, since $p \equiv 1 \pmod{4}$, we have $\left( \dfrac{p}{q} \right)=\left( \dfrac{q}{p} \right)$
by quadratic reciprocity. So we require $\left( \dfrac{p}{q} \right)=-1$ by
Lemma~\ref{lem:T2} in order that $C_{p}$ is not everywhere locally solvable.

For $C_{2p}$, by Lemma~\ref{lem:T3}, we need $\left( \dfrac{2p}{q} \right)=-1$
or $\left( \dfrac{2q}{p} \right)=-1$ or $p+q \not\equiv 0,2,8 \pmod{16}$ so
that $C_{2p}$ is not everywhere locally solvable.

We now combine these conditions with those obtained from $C_{2}$ and $C_{p}$.

Suppose that $p \equiv 5 \pmod{8}$ (one of the possibilities from $C_{2}$). By
quadratic reciprocity and Lemma~\ref{lem:T2}, we
have $\left( \dfrac{p}{q} \right)=\left( \dfrac{q}{p} \right)=-1$. Hence
$\left( \dfrac{2q}{p} \right)=-\left( \dfrac{2}{p} \right)=1$ always holds.
We can eliminate $q \equiv 3 \pmod{8}$, since $p+q \equiv 0 \pmod{8}$ and
$\left( \dfrac{2p}{q} \right)=-\left( \dfrac{2}{q} \right)=1$, so by
Lemma~\ref{lem:T3}, $C_{2p}$ is everywhere locally solvable.
This leaves only $q \equiv 7 \pmod{8}$, so by Lemma~\ref{lem:T3} again, we
have $\left( \dfrac{2p}{q} \right)=-\left( \dfrac{2}{q} \right)=-1$. Therefore,
when $p \equiv 5 \pmod{8}$, $C_{2p}$ is not everywhere locally solvable if and
only if $q \equiv 7 \pmod{8}$. Equivalently, $pq \equiv 3 \pmod{8}$.

Next, suppose that $q \equiv 3 \pmod{8}$ (another possibility from $C_{2}$). By
the same reasoning, we find that
$C_{2p}$ is not everywhere locally solvable if and only if $p \equiv 1 \pmod{8}$.
Here again, we have $pq \equiv 3 \pmod{8}$.

Now suppose that $pq \not\equiv 7 \pmod{8}$ (the remaining condition arising from $C_{2}$).
The possibilities are $(p,q) \equiv
(1,3), (5,7) \pmod{8}$. For the first pair, $\left( \dfrac{2}{p} \right)=1$, so
$\left( \dfrac{2q}{p} \right)=-\left( \dfrac{2}{p} \right)=-1$ (recall that
$\left( \dfrac{q}{p} \right)=-1$ by our consideration of $C_{p}$). Hence
$C_{2p}$ is not everywhere locally solvable. Similarly, for the second pair,
$C_{2p}$ is not everywhere locally solvable when $pq \not\equiv 7 \pmod{8}$.

$\widehat{C}_{-1}$ is never everywhere locally solvable since the condition
$p \equiv q \equiv 1 \pmod{4}$ in Lemma~\ref{lem:oT1} does not hold here.

For $\widehat{C}_{p}$, $\left( \dfrac{p}{q} \right) \neq 1$ by our consideration
of $C_{p}$. So $\widehat{C}_{p}$ is not everywhere locally solvable by
Lemma~\ref{lem:oT3}.

Similarly, $\widehat{C}_{-p}$ is also not everywhere locally solvable by
Lemma~\ref{lem:oT4}, since $\left( \dfrac{q}{p} \right)
=\left( \dfrac{p}{q} \right)=-1$.

\vspace*{2.0mm}

\noindent
$\bullet$ Suppose that $p \equiv 3 \pmod{4}$ and $q \equiv 1 \pmod{4}$.

$C_{2}$ is not everywhere locally solvable if and only if $p \equiv 3 \pmod{8}$
or $q \equiv 5 \pmod{8}$ or $pq \equiv 3 \pmod{8}$, by
Lemma~\ref{lem:T1}.

$C_{p}$ is not everywhere locally solvable if and only if
$\left( \dfrac{p}{q} \right)=\left( \dfrac{q}{p} \right)=-1$ (applying quadratic
reciprocity), by Lemma~\ref{lem:T2}.

$C_{2p}$ is not everywhere locally solvable if and only if
$\left( \dfrac{2p}{q} \right)=-\left( \dfrac{2}{q} \right)=-1$,
or $\left( \dfrac{2q}{p} \right)=-\left( \dfrac{2}{p} \right)=-1$
or $p+q \not\equiv 0,2,8 \pmod{16}$. Here we use
$\left( \dfrac{p}{q} \right)=\left( \dfrac{q}{p} \right)=-1$ from $C_{p}$.

If the condition $p \equiv 3 \pmod{8}$ holds, then $p+q \not\equiv 0 \pmod{8}$ holds
if and only if $q \equiv 1 \pmod{8}$, while $p+q \equiv 2 \pmod{16}$ never
holds (since $q \equiv 1 \pmod{4}$). Hence the condition $pq \equiv 3 \pmod{8}$
holds.

Similarly, if $q \equiv 5 \pmod{8}$, then we must have $p \equiv 7 \pmod{8}$
and again $pq \equiv 3 \pmod{8}$.

Lastly, suppose that $pq \equiv 3 \pmod{8}$. This is only possible when
$(p,q) \equiv (3,1), (7,5) \pmod{8}$. In the first case, $\left( \dfrac{2p}{q} \right)
=-\left( \dfrac{2}{q} \right)=-1$, so $C_{2p}$ is not everywhere locally solvable.
In the second case, $\left( \dfrac{2q}{p} \right)=-1$, so we also find that
$C_{2p}$ is not everywhere locally solvable.

$\widehat{C}_{-1}$ is never everywhere locally solvable since $p \equiv 3 \pmod{4}$.

$\widehat{C}_{p}$ and $\widehat{C}_{-p}$ are both never everywhere locally solvable
due to the condition $\left( \dfrac{p}{q} \right)=\left( \dfrac{q}{p} \right)=-1$
obtained from considering $C_{p}$ and applying Lemmas~\ref{lem:oT3} and \ref{lem:oT4}.

\vspace*{2.0mm}

\noindent
$\bullet$ Suppose that $p \equiv q \equiv 3 \pmod{4}$.

$C_{2}$ is not everywhere locally solvable if and only if $p \equiv 3 \pmod{8}$
or $q \equiv 3 \pmod{8}$ or $pq \not\equiv 1 \pmod{16}$, by Lemma~\ref{lem:T1}.
Since $p \equiv q \equiv 3 \pmod{4}$, we cannot have $pq \equiv 7 \pmod{8}$,
which is one of the conditions in Lemma~\ref{lem:T1}.

$C_{p}$ is never everywhere locally solvable since neither $p \equiv 1 \pmod{4}$
nor $q \equiv 1 \pmod{4}$ holds.

$C_{2p}$ is not everywhere locally solvable if and only if
$\left( \dfrac{2p}{q} \right)=-1$ or $\left( \dfrac{2q}{p} \right)=-1$ or
$p+q \not\equiv 0,2,8 \pmod{16}$.

$\widehat{C}_{-1}$ is never everywhere locally solvable since $p \equiv q \equiv 3 \pmod{4}$.

$\widehat{C}_{p}$ is not everywhere locally solvable if and only if
$\left( \dfrac{-q}{p} \right)=-\left( \dfrac{q}{p} \right)=-1$ or
$\left( \dfrac{p}{q} \right)=-1$ or
all of $p \not\equiv 1 \pmod{8}$, $q \not\equiv 7 \pmod{8}$ and
$p-q \not\equiv 0,4 \pmod{16}$ hold.

$\widehat{C}_{-p}$ is not everywhere locally solvable if and only if
$\left( \dfrac{q}{p} \right)=-1$ or
$\left( \dfrac{-p}{q} \right)=-\left( \dfrac{p}{q} \right)=-1$ or
all of $p \not\equiv 7 \pmod{8}$, $q \not\equiv 1 \pmod{8}$ and
$q-p \not\equiv 0,4 \pmod{16}$ hold.

We now combine this information to prove the theorem in this case.

First suppose that $\left( \dfrac{p}{q} \right)=1$. Since $p \equiv q \equiv 3 \pmod{4}$,
we also have $\left( \dfrac{q}{p} \right)=-1$ and $\left( \dfrac{-q}{p} \right)=1$.
So in this case, $\widehat{C}_{-p}$ is not everywhere locally solvable.

Also $\widehat{C}_{p}$ is not everywhere locally solvable if and only
if $q \equiv 3 \pmod{8}$ and $p-q \not\equiv 0,4 \pmod{16}$ (i.e., when
$p-q \equiv 8,12 \pmod{16}$). This is condition~(iii) in the theorem.

Since $q \equiv 3 \pmod{8}$, we have seen that $C_{2}$ is not everywhere locally solvable.

Also, from $q \equiv 3 \pmod{8}$, we have $\left( \frac{2p}{q} \right)
=\left( \frac{2}{q} \right)\left( \frac{p}{q} \right)=-1$, so $C_{2p}$ is also
not everywhere locally solvable.

Similarly, if $\left( \dfrac{p}{q} \right)=-1$, then we can use Lemma~\ref{lem:oT4}
to show that $\widehat{C}_{-p}$ is not everywhere locally solvable if and only
if condition~(iv) in the theorem holds.

\section{Proof of Theorem~\ref{thm:rk0-sha1}}
\label{sect:thm-rk0-sha1}

That the conditions in Theorem~\ref{thm:rk0} imply the statement about the rank
of the Mordell-Weil group and the $2$-torsion subgroup of the Tate-Shafarevich
group follows immediately from Theorem~\ref{thm:rk0} and Proposition~\ref{prop:2.1}
since $\dim_{2} S^{(\phi)}\left( E_{p,q}/\bbQ \right)
=\dim_{2} S^{(\widehat{\phi})}\left( \widehat{E_{p,q}}/\bbQ \right)=1$ and
$\dim_{2} \cC_{S} \geq 0$.

For the other direction, we know that $\widehat{E_{p,q}}(\bbQ)[2] \cong \bbZ/2\bbZ$
from Proposition~X.6.1(a) of \cite{S09}. From this along with parts~(b) and (c)
of Lemma~\ref{lem:Cphi-2}, we see that $\dim_{2} \cC_{S}=0$ when
$\dim_{2} S^{(\widehat{\phi})}\left( \widehat{E_{p,q}}/\bbQ \right) \leq 2$.
Hence from Proposition~\ref{prop:2.1}, we cannot have
$\rank E(\bbQ)+\dim_{2} \Sha (E/\bbQ)[2]=0$ in this case unless
$\dim_{2} S^{(\phi)}\left( E_{p,q}/\bbQ \right)
=\dim_{2} S^{(\widehat{\phi})}\left( \widehat{E_{p,q}}/\bbQ \right)=1$.

The other possibility is that
$\dim_{2} S^{(\widehat{\phi})}\left( \widehat{E_{p,q}}/\bbQ \right)=3$.
The reason is that Lemma~\ref{lem:oT2} removes all square
classes with a factor of $2$, so $S^{(\widehat{\phi})}\left( \widehat{E_{p,q}}/\bbQ \right)$
is contained in the subgroup generated by $-1$, $p$ and $q$, namely
$\{ 1, -1, p, -p, q, -q, pq, -pq \}$.

In this case, $-1, p \in S^{(\widehat{\phi})}\left( \widehat{E_{p,q}}/\bbQ \right)$,
so by Lemmas~\ref{lem:oT1} and \ref{lem:oT3}, we have $p \equiv q \equiv 1 \pmod{4}$
and $\left( \dfrac{p}{q} \right)=\left( \dfrac{q}{p} \right)=1$. By Lemma~\ref{lem:T2},
$p,q \in S^{(\phi)} \left( E_{p,q}/\bbQ \right)$, so $\dim_{2} S^{(\phi)} \left( E_{p,q}/\bbQ \right)
\geq 2$. By Lemma~\ref{lem:Cphi-1} and since
$\dim_{2} \cC_{S} \leq \dim_{2} S^{(\widehat{\phi})}\left( \widehat{E_{p,q}}/\bbQ \right)=3$,
we also have $\dim_{2} \cC_{S} \in \{ 0,2 \}$.
Combining these facts with Proposition~\ref{prop:2.1}, we find that
we cannot have
$\rank E(\bbQ)+\dim_{2} \Sha (E/\bbQ)[2]=0$ in this case.

So if $\rank E(\bbQ)+\dim_{2} \Sha (E/\bbQ)[2]=0$, then we must have
$\dim_{2} S^{(\phi)}\left( E_{p,q}/\bbQ \right)
=\dim_{2} S^{(\widehat{\phi})}\left( \widehat{E_{p,q}}/\bbQ \right)=1$. By
Theorem~\ref{thm:rk0}, one of the conditions in Theorem~\ref{thm:rk0} must hold.
This shows that the converse in Theorem~\ref{thm:rk0-sha1} holds, completing its
proof.

\section{Proof of Corollary~\ref{cor:rk0-Qi}}
\label{sect:cor-rk0-Qi}

We start with a result about the rank of quadratic twists of an elliptic curve $E$.

If $\bbK = \bbQ \left( \sqrt{m} \right)$, where $m$ is a squarefree integer,
then the rank of any elliptic curve $E$ over $\bbK$ is the sum of the ranks of $E$
and its quadratic twist $E^{(m)}$ by $m$ over $\bbQ$. This can be seen from the
following result (see \cite[Exercise~10.22(c)(iv)]{S09}).

\begin{proposition}
\label{prop:2.2}
Let $\bbK = \bbQ \left( \sqrt{m} \right)$ be a quadratic field, where $m \neq 1$
is a squarefree integer. Let $E: y^{2} = x^{3} + ax^{2} + bx +c$ be an elliptic
curve over $\bbQ$ and $E^{(m)}: y^{2} = x^{3} + max^{2} + m^{2} bx + m^{3}c$ be
the quadratic twist of $E$ by $m$. Then
\[
\rank E(\bbK) = \rank E(\bbQ) + \rank E^{(m)}(\bbQ).
\]
\end{proposition}

\begin{corollary}
\label{cor:rank-i}
Let $\bbK = \bbQ(i)$ and $E_{D}: y^{2}=x^{3}+Dx$ for
nonzero $D \in \bbZ$. Then
\begin{equation}
\label{eq:rank-minus1-twist}
\rank E_{D}(\bbK) = 2\rank E_{D}(\bbQ).
\end{equation}
\end{corollary}

\begin{proof}[Proof of Corollary~\ref{cor:rank-i}]
We apply Proposition~\ref{prop:2.2} with $m=-1$, $\bbK=\bbQ(i)$ and $E=E_{D}$.
Hence $E_{D}^{(-1)}: y^{2} =x^{3}+(-1)^{2}Dx= x^{3}+Dx= E_{D}$.
Using Proposition~\ref{prop:2.2}, we have
\[
\rank E_{D}(\bbK)
= \rank E_{D}(\bbQ) + \rank E_{D}^{(-1)}(\bbQ)
= 2\rank E_{D}(\bbQ).
\]
\end{proof}

\begin{proof}[Proof of Corollary~\ref{cor:rk0-Qi}]
From Theorem~\ref{thm:rk0-sha1}, we know that $\rank E_{p,q}(\bbQ)=0$.
So our claim follows directly from \eqref{eq:rank-minus1-twist} with $D=-pq$.
\end{proof}

\end{document}